\documentclass[a4]{amsart}
\usepackage{amssymb}
\usepackage{amsmath}
\usepackage{amscd}
\usepackage{verbatim,ifthen}
\usepackage{graphicx}
\usepackage{xypic}
\usepackage{cite}
\usepackage{amsthm}
\usepackage{todonotes}
\usepackage{tikz-cd}

\usepackage{latexsym}

\newtheorem*{lemma*}{Lemma}

\newcommand\NN{\mathbb{N}}
\newcommand\RR{\mathbb{R}}
\newcommand\CC{\mathbb{C}}

\newcommand\A{A} %{\mathcal{A}} Update this one
\newcommand\saA{\A^\text{sa}}
\newcommand\cA{\A_\text{c}}
\newcommand\B{\mathcal{B}^\A}

\newcommand\mA{M(\A)}%{\mathcal{M}(\A)} %multiplier algebra. Update this one
\newcommand\samA{\mA^\text{sa}}

\newcommand\mB{\mathcal{B}^{\mA}}

\newcommand\evnA{\A''} %enveloping von Neumann algebra
\newcommand\Proj{\mathcal{P}\left(\evnA\right)} %projections
\newcommand\cProj{\overline{\mathcal{P}}\left(\A\right)} %closed projections
\newcommand\oProj{\mathring{\mathcal{P}}\left(\A\right)} %open projections

\newcommand\supp{\operatorname{supp}} %support projection
\newcommand\csupp{\overline{\operatorname{supp}}} %closed support projection

\newcommand\dilup{\Lambda_L}

\newcommand\dilupsum{\Lambda^1_L}

\newcommand\dilupsup{\Lambda^\infty_L}

\newcommand{\simp}[1]{\Delta^#1}
\newcommand{\osimp}[1]{\mathring{\Delta}^#1}
\newcommand{\uppreimage}[1]{{#1}^{-1}_\geq}

\newcommand{\locflat}[1]{\mathcal{B}^a_0}

\newcommand\States{\mathcal{S}(\A)}

\newcommand\mStates{\mathcal{S}(\mA)}
\newcommand\cStates{\mathcal{S}_\text{c}(\A)}
\newcommand\LcomStates{\mathcal{S}_\text{c}^L(\A)}

\newcommand{\Code}[2]{\Sigma_{#1}({#2})}
\newcommand{\infCode}[1]{\Sigma_{#1}}

\newcommand\dLto{\xrightarrow{d_L}}
\newcommand\wsto{\xrightarrow{w*}}

\newtheorem{thm}{Theorem}[section]

\newtheorem{lemma}[thm]{Lemma}
\newtheorem{cor}[thm]{Corollary}%*

\theoremstyle{definition}
\newtheorem{defn}[thm]{Definition}

\newtheorem{conj}[thm]{Conjecture}

\newcommand{\noop}[1]{}

\usepackage{hyperref}

\begin{document}

\title{Self-similar states and projections in noncommutative metric spaces}

\author{Sean Harris}
\address{Hanna Neumann Building \#145, Science Road
The Australian National University
Canberra ACT 2601.}
\email{Sean.Harris@anu.edu.au}     

\date{\today}

\thanks{The author gratefully acknowledges financial support by the Statistical Science Research Endowment and the MSI at the Australian National University, and ARC Discovery Project DP160100941.}

\begin{abstract}
We present a generalisation of the theory of iterated function systems and associated fractals to the setting of noncommutative geometry. Along the way, we discuss some ideas surrounding locally compact noncommutative metric spaces.
\end{abstract}

 
 \maketitle

\section{Classical IFS theory} %, extended metric spaces}
 
Before delving into the noncommutative world, we briefly explain the fundamentals of classical iterated function system (IFS) and fractal theory (see \cite{Hutchinson} for precise statements and proofs).

Let $(X,d)$ be a complete metric space. An iterated function system (IFS) on $X$ is a collection of functions $\mathbb{G} = \left(g_i:X \to X\right)_{i=1}^k$. The IFS $\mathbb{G}$ is said to be \emph{strictly contractive} if the maximum of the Lipschitz constants 
\[\Lambda_d(g_i) = \sup_{x \neq y} \frac{d(g(x),g(y))}{d(x,y)}\]
of the maps composing $\mathbb{G}$ is less than $1$.

Given a strictly contractive IFS $\mathbb{G}$, there is a unique non-empty compact subset $K_\mathbb{G} \subset X$ satisfying the following equation
\[K_\mathbb{G} = \bigcup_{i=1}^k g_i\left(K_\mathbb{G}\right),\]
which is known as the \emph{attractor} of $\mathbb{G}$. The above equation demonstrates self-similarity of $K_\mathbb{G}$, a set which is often - but not always - ``fractal" in nature (fractional dimension, boundary of positive measure, etc.).

Alternatively, given a weight $\pi = (\pi_i)_{i=1}^k$ (i.e a $k$-tuple of non-negative numbers summing to $1$), there exists a unique compactly-supported regular Borel probability measure $\mu$ on $X$ such that
\[\mu = \sum_{i=1}^k \pi_i g_i^*\mu,\]
where $g_i^*\mu$ is the pushforward of $\mu$ by $g_i$. Providing none of the weights in $\pi$ is $0$, the support of $\mu$ is exactly $K_\mathbb{G}$.

In fact, one obtains a (typically not compactly-supported) self-similar measure as above even if the maps composing $\mathbb{G}$ are only contractive ``on average" (see \cite{Elton1987} and references therein).

There are many representation formulas for both the attractor and the self-similar measures arising from a strictly contractive IFS, along with related ergodic theorems. For example, the attractor of a strictly contractive IFS $\mathbb{G}$ is the closure of the fixed points of all finite compositions of maps in $\mathbb{G}$, and the associated self-similar measures are weak-$*$ limits of certain convex combinations of Dirac masses of said fixed points.

Of particular interest to harmonic analysis are self-similar tilings. Existence of self-similar tilings relies on the theory of IFS and fractals. For example, the well-known Haar decomposition (otherwise known as the dyadic decomposition) of function spaces on $\mathbb{R}^n$ decomposes a function into pieces whose supports lie within dyadic boxes. Alternatively, the Littlewood-Paley decomposition decomposes a function on $\mathbb{R}^n$ into pieces whose Fourier transforms are supported on a tiling of the Pontryagin dual of $\mathbb{R}^n$ (excluding the origin) by dyadic annuli.

Both the Haar and Littlewood-Paley decompositions are indispensable tools of Euclidean harmonic analysis, and much of their usefulness is related to their self-similar nature.

Suppose one is instead interested in harmonic analysis on a (non-Abelian) locally compact group $G$. Analogues of the Haar decomposition in such a setting have been studied (see Chapter 3 of \cite{myPhDThesis} and references therein). However, there is no longer a Pontryagin dual to support an analogue of the Littlewood-Paley decomposition. In its place is instead the noncommutative topological space associated with $C^*(G)$, the group $C^*$-algebra of $G$. Thus, one should look for self-similar ``tilings" of $C^*(G)$.

This article provides an extension of IFS and fractal theory and results to the setting of noncommutative geometry. It is hoped that such a theory will be a first step in making the above heuristic a reality.

\section{Spectral metric spaces}\label{secMetricSpaces}
In order to extend the classical theory of IFS to noncommutative geometry, we need metric information. For this, we follow the $C^*$-algebraic approach of Rieffel, Latr{\'e}moli{\`e}re, etc. (see \cite{Rieffel2003}, \cite{Latremoliere2013}), rather than the von Neumann algebraic approach of Kuperberg and Weaver (see \cite{KuperbergWeaver2010}).

Let $\A$ be a $C^*$-algebra, and let $L:\saA \to [0, \infty]$ be an extended semi-norm on the self-adjoint part of $\A$ such that

\begin{enumerate}

\item \label{propertyDense} The set $\B := \{c \in \saA; L(c)<\infty\}$ is norm dense in $\saA$.

%\item The set $\{b \in \saA; L(b) < 1\}$ contains an approximate identity for $\A$.

\item \label{propertyAI} There is an approximate identity $(I_n)_{n \in N} \subset \A$ with $L(I_n) \to 0$.

\item \label{propertyLSC} $L$ is $||\cdot||$-lower semi-continuous, i.e. for any $c \geq 0$, $\{b \in \saA; L(b) \leq c\}$ is the $||\cdot||$-closure of $\{b \in \saA; L(b) < c\}$.

\end{enumerate}

We call such a pair $(\A, L)$ an extended complete spectral metric space.

The classical picture one should have in mind is that $L$ is the Lipschitz semi-norm on the set of real valued functions on a locally compact (extended) metric space vanishing at infinity (in the topological sense rather than the metric sense). 

We let $\mA$ denote the multiplier algebra of $\A$ (for details regarding $\mA$, see \cite{PedersenCstarAndAutomorphisms1979}), and will consider in the usual way $\A$ and $\mA$ as subalgebras of the universal enveloping von Neumann algebra $\evnA$ of $\A$ (which we identify with the second dual of $\A$). Denote by $\States$ and $\mStates$ the sets of states of $\A$ and $\mA$ respectively. We will identify $\States$ with the  subset of $\mStates$ consisting of those states which are strictly continuous. For $R>0$ we let $\A_R, \mA_R$ denote the closed unit balls of $\A$, $\mA$ respectively, and set

\[\B_R := \{c \in \saA; L(c) \leq R\}.\]

We define an extended distance function $d_L: \States \times \States \to [0, \infty]$ by the formula

\begin{equation}\label{eqndL}
d_L(\phi, \psi) = \sup_{b \in \B_1} |\phi(b)-\psi(b)|.
\end{equation}

Equivalently,

\[d_L(\phi, \psi) = \inf\{C>0; ~ |\phi(b)-\psi(b)| \leq C L(b) \text{ for all } b \in \saA\}.\]

The extended distance function $d_L$ will be referred to as the spectral distance. In the case that $\A = C_0(X)$ for some locally compact extended metric space $(X,d)$ and $L$ is taken to be the usual Lipschitz semi-norm induced by $d$, then the inclusion of $X$ into $\States$ given by sending $x \in X$ to the evaluation functional $\text{ev}_x$ is an isometry, i.e. $d(x,y) = d_L(\text{ev}_x, \text{ev}_y)$.

Note that $d_L$ may take the value $\infty$. Much work has been done in the area of noncommutative metric spaces to avoid these infinite distances, although some natural examples exhibit infinitely distant states. Thus, we will instead make no further restrictions at this point, which comes at the price of many of our theorem statements including conditions related to infinite distances. We will re-examine this topic in Subsection \ref{subsecTopology}.

We will assume that
\begin{equation}\label{eqnLConsistent}
L(b) = \inf\{C>0; ~ |\phi(b)-\psi(b)| \leq C d_L(\phi, \psi), ~\phi,\psi \in \States\}.
\end{equation}
We can do so without loss of generality for the following reason: given any $L$ satisfying our three requisite properties (even just the first two), one may create a $\tilde{L}$ satisfying all three properties such that $d_L = d_{\tilde{L}}$, with $\tilde{L}$ also satisfying the above equality by simply defining $\tilde{L}(b)$ to be the quantity appearing on the right in Equation (\ref{eqnLConsistent}). Norm lower-semicontinuity of $\tilde{L}$ follows as the pointwise supremum of the norm continuous functions $b \mapsto |\phi(b)-\psi(b)|/d_L(\phi,\psi)$, for $\phi \neq \psi \in \States$. Note that for any $b \in \saA$, $\tilde{L}(b) \leq L(b)$.

We will also need to consider a natural extension of $L$ to $\samA$, the self-adjoint part of the multiplier algebra of $\A$, which we still denote $L$. This is defined by exactly the same formula as in Equation \ref{eqnLConsistent}, but for $b \in \samA$. Note that $L$ is lower semicontinuous in the strict topology on $\samA$, as the functions $b \mapsto |\phi(b)-\psi(b)|/d_L(\phi, \psi)$ are strictly continuous (since the strictly continuous linear functionals on $\mA$ are exactly the (unique extensions) of norm continuous linear functionals on $\A$). We denote by $\mB$ the set of $b \in \samA$ with $L(a)<\infty$, and $\mB_R$ the (strictly closed) set of elements $b \in \samA$ with $L(b)\leq R$.

Under the extension of $L$ to $\samA$, we similarly extend $d_L$ to an extended distance function on all of $\mStates$ via the formulas
\[d_L(\phi, \psi) = \sup_{b \in \mB_1} |\phi(b)-\psi(b)| = \inf\{C>0; ~ |\phi(b)-\psi(b)| \leq C L(b) \text{ for all } b \in \samA\}\]

Note that by strict density of $\A$ in $\mA$, strict continuity of (the extensions) of states in $\States$, and strict lower-semicontinuity of the extension of $L$ to $\samA$, the above formula for $d_L$ agrees with the previous when restricted to $\phi, \psi \in \States$.

The following is an immediate consequence of Property \ref{propertyDense}.

\begin{lemma}\label{lemmadLtoImplieswsto}
If $\phi_n \dLto \phi$, then $\phi_n \wsto \phi$
\end{lemma}

While the next two results are consequences of Property \ref{propertyAI} (similar to work of Mesland and Rennie in \cite{MeslandRennie2016}).

\begin{lemma}
For any states $\phi \in \States$ and $\psi \in \mStates \backslash \States$, $d_L(\phi, \psi)=\infty$.
\end{lemma}

\begin{proof}
We have assumed that there is an approximate identity $(I_n)_{n \in N} \subset \A$ with $L(I_n) \to 0$. Then $\phi(I_n) \to 1$, while $\limsup \psi(I_n) < 1$ since $\psi$ is not a state when restricted to $\A$. Hence $\liminf |\phi(I_n)-\psi(I_n)|>0$, while $L(I_n) \to 0$. Rearranging, we find $d_L(\phi, \psi)=\infty$.
\end{proof}

In English, the previous lemma asserts that topological infinity is infinitely distant, as one should hope.

\begin{lemma}\label{lemmaComplete}
Let $(\A, L)$ be an extended complete spectral metric space. Then $\left(\States, d_L\right)$ is complete.
\end{lemma}

\begin{proof}
Suppose that $\{\phi_n\}_{n=1}^\infty \subset \States$ is $d_L$-Cauchy. By weak-$*$ compactness of $\mStates$ (the state space of a unital $C^*$-algebra), we may pass to a subsequence and assume that $\phi_n \wsto \phi \in \mStates$. For any $\epsilon>0$ we can pick $M$ such that for all $m, n>M$ and $b \in \mB_1$ we have
\[\left|\phi_n(b) - \phi_m(b)\right| < \epsilon.\]
Taking $m \to \infty$ and using the fact that $\phi_m \wsto \phi$, this implies that for all $b \in \mB_1$, $n>M$,
\[\left|\phi_n(b) - \phi(b)\right| < \epsilon.\]
So $\phi_n \dLto \phi \in \mStates$. By applying the previous lemma we can deduce that $\phi \in \States$, so $(\States, d_L)$ is complete. 
\end{proof}

%%%%%%%%%%%%%%%%%%%%%%%%%%%%%%%%%%%%%%%%%%%%%%%%%%%%%%%%%%%%%%%%%%%%%%%%%%%%%%%%%%%%%%%%%%%%%%%%%%%%%%%%%%%%%%%%%%%%%%%%%%%%%%%%%%%%%%%%%%%%%%%%%%%%%%%%%%%%%%%%%%%%%%%%%%%%%%%%%%%%%%%%%%%%%%%%%%%%%%%%%%%%%%%%%%%%%%%%%%%%%%%%%%%%%%%%%%%%%%%%%%%%%%%%%%%%%%%%%%%%%%%%%%%%%%%%%%%%%%%%%%%%%%%%%%%%%%%%%%%%%%%%%%%%%%%%%%%%%%%%%%%%%%%%%%%%%%%%%%%%%%%%%%%%%%%%%%%%%%%%%%%%%%%%%%%%%%%%%%%%%%%%%%%%%%%%%%%%%%%%%%%%%%%%%%%%%%%%%%%%%%%%%%%%%%%%%%%%%%%%%%%%%%%%%%%%%%%%%%%%%%%%%%%%

\section{Dual IFS, existence and uniqueness of self-similar states}

In generalising classical IFS to the noncommutative setting, one may naturally think to work with a finite collection of proper $*$-endomomorphisms of a $C^*$-algebra $A$ (i.e. $*$-homomorphisms which send approximate units to approximate units). However, this turns out to be too restrictive a setting to capture even the commutative state of affairs. Indeed, if $A = C_0(X)$ for some locally compact Hausdorff space $X$, then proper $*$-endomorphisms of $A$ are exactly precompositions with \emph{proper} continuous functions from $X$ to itself (proper meaning that the pre-image of each compact set is again compact). For example, this would eliminate IFS in which any of the maps have compact image (providing $X$ itself is not compact). 

We thus wish to allow for ``non-proper" maps. For the noncommutative setting, we have the following definition:

\begin{defn}
A $*$-homomorphism $f:A \to M(B)$ is called relatively proper if there is an increasing approximate identity $(I_n)_{n \in N}$ such that $f(I_n) \to I \in B$ strictly (i.e. for any $b \in B \subset M(B)$, $||f(I_n)b - b||, ~||bf(I_n)-b|| \to 0$).
\end{defn}

Note that relatively proper $*$-homomorphisms from $A$ to $M(B)$ are exactly Lance's ``morphisms" from $A$ to $B$, which have been introduced in the study of Hilbert $C^*$-modules (see \cite{LanceHilbCModules}). Lance provides the following knowledge regarding their structure (simplified for our setting, in which Lance's $E$ is taken to be $\mathcal{K}(B) \cong B$, for which $\mathcal{L}(E) \cong M(B)$).

\begin{thm}[\cite{LanceHilbCModules}, Prop 2.1, 2.5]\label{thmLanceMorphism}
Suppose $f:A \to M(B)$ is a $*$-homomorphism. The following are equivalent:
\begin{itemize}
\item $f$ is relatively proper.
\item $f$ is the restriction of a (unique) unital $*$-homomorphism $M(A) \to M(B)$ which is strictly continuous on the unit ball.
\end{itemize}
\end{thm}

We will still use $f$ to denote the unique extension guaranteed by the second point above. Note that the second point also implies that given relatively proper $*$-homomorphisms $f:A \to M(B)$ and $g:B \to M(C)$, one may compose the unique extensions of $f$ and $g$ to obtain $g \circ f:M(A) \to M(C)$ which is still strictly continuous on the unit ball and hence restricts to a relatively proper $*$-homomorphism from $A$ to $M(C)$. We will compose such morphisms in this way in future.

Relative properness ensures that the dual map $f^*: M(B)^* \to A^*$ sends (strictly continuous extensions of) states on $B$ to states on $A$, as is verified below.

\begin{lemma}\label{lemmaAdjointStatetoState}
Suppose $f:A \to M(B)$ is a relatively proper $*$-homomorphism, and $\phi \in \mathcal{S}(B)$. Then $f^*(\phi) \in \mathcal{S}(A)$.
\end{lemma}

\begin{proof}
The only possible issue is that the norm of $f^*(\phi)$ is less than $1$. Pick an increasing approximate identity $(I_n)_{n \in N}$ for $A$ as in the definition of relative properness. Noting $\mathcal{S}(B) \subset \mathcal{S}(M(B))$ are exactly the strictly continuous states, we have

\[\lim_{n \in N}f^*(\phi)(I_n) = \lim_{n \in N}\phi(f(I_n))  = \phi(\lim_{n \in N}f(I_n)) = \phi(I) = 1.\]

So $||f^*(\phi)||=1$.
\end{proof}

\begin{defn}
A dual iterated function system (dual IFS) on a $C^*$-algebra $\A$ is a finite list $\mathbb{F} = (f_1, \ldots, f_k)$ of relatively proper $*$-homomorphisms $f_i:\A \to \mA$.
\end{defn}

We denote the standard $(k-1)$-simplex by
\[\simp{k} = \{(\pi_1, \ldots, \pi_k) \in \RR^k; \pi_1, \ldots, \pi_k \in [0,1], \pi_1+\ldots+\pi_k =1\}\]
and the strict $(k-1)$-simplex by
\[\osimp{k} = \{(\pi_1, \ldots, \pi_k) \in \RR^k; \pi_1, \ldots, \pi_k \in (0,1), \pi_1+\ldots+\pi_k =1\}.\]

\begin{defn}
Given a dual IFS $\mathbb{F} = (f_1, \ldots, f_k)$ and $\pi = (\pi_1, \ldots, \pi_k) \in \simp{k}$, define $\pi \cdot \mathbb{F}^*:\mA^* \to \A^*$ by
\[(\pi\cdot\mathbb{F}^*)\phi = \sum_{i=1}^k \pi_i f_i^*\phi = \sum_{i=1}^k \pi_i \phi\circ f_i.\]
\end{defn}

Since dual IFS consist of relatively proper $*$-homomorphisms, $\pi\cdot\mathbb{F}^*$ restricts to a map $\States \to \States$.

We set $\Code{k}{M} = \{1, \ldots, k\}^M$ to be the length $M$ \emph{code space} on $k$ symbols. Given $\pi \in \simp{k}$, a dual IFS $\mathbb{F} = (f_1, \ldots, f_k)$, and $(\omega_1, \ldots, \omega_M) = \omega \in \Code{k}{M}$, we set
\[\pi_\omega = \prod_{m=1}^M \pi_{\omega_m},\]
and
\[f_\omega = f_{\omega_1} \circ \ldots \circ f_{\omega_M}.\]
Note that $\sum_{\omega \in \Code{k}{M} } \pi_\omega = 1$. A quick computation verifies that
\[(\pi\cdot\mathbb{F}^*)^M\phi = \sum_{\omega \in \Code{k}{M}} \pi_\omega f_\omega^*\phi\]
for each $M$. We will make use of this fact throughout the paper.

\begin{defn}
A state $\phi \in \States$ satisfying
\[\phi = (\pi\cdot\mathbb{F}^*)(\phi)\]
is called $(\pi\cdot\mathbb{F}^*)$-self-similar, or just self-similar when $\pi$ and $\mathbb{F}$ are clear from context.
\end{defn}

\begin{defn}
Given an extended complete spectral metric space $(\A, L)$ and an endomorphism $f$ on $\A$, the upper dilation factor $\dilup(f) \in [0, \infty]$ is defined by
\[\dilup(f) = \sup\{L(f(b)); L(b) \leq 1\}.\]
%and the lower dilation factor $\dillow(f) \in [0, \infty]$ by
%\[\dillow(f) = \inf\{L(f(b)); L(b) \geq 1\}.\]

Given a $\pi \in \simp{k}$ and a dual IFS $\mathbb{F} = (f_1, \ldots, f_k)$ on $\A$, we likewise define the $1$- and $\infty$- dilation factors of $\Pi\cdot\mathbb{F}^*$ by
\[\dilupsum(\pi\cdot\mathbb{F}^*) = \sum_{i=1}^k \pi_i \dilup(f_i) \in [0, \infty],\]
and
\[\dilupsup(\mathbb{F}) = \max_{i=1, \ldots, n} \dilup(f_i) \in [0, \infty].\]
We say a dual IFS $\mathbb{F}$ is strictly contractive if $\dilupsup(\mathbb{F})<1$.
%\[\dillowsum(\pi\cdot\mathbb{F}^*) = \sum_{i=1}^k \pi_i \dillow(f_i) \in [0,\infty].\]
\end{defn}

Note that given some $C\in [0,\infty]$ and dual IFS $\mathbb{F}$, the set of $\pi \in \simp{k}$ for which $\dilupsum( \pi\cdot\mathbb{F}^*) < C$ is a (possibly empty) simplex, and is hence (path) connected.

\begin{lemma}\label{lemmaLip}
Let $(\A, L)$ be an extended complete spectral metric space, $\mathbb{F} = (f_1, \ldots, f_k)$ a dual IFS on $\A$ and $\pi \in \simp{k}$. Then for any $\phi, \psi \in \States$,
\[d_L\left((\pi\cdot\mathbb{F}^*)(\phi), (\pi\cdot\mathbb{F}^*)(\psi)\right) \leq \dilupsum(\pi\cdot\mathbb{F}^*)d_L(\psi, \psi).\]
\end{lemma}

\begin{proof}
Given $\phi, \psi \in \States$ and $b \in \B_1$, we have
\begin{align*}
\left| (\pi\cdot\mathbb{F}^*)(\phi)(b) - (\pi\cdot\mathbb{F}^*)(\psi)(b) \right| &= \left| \sum_{i=1}^k \pi_i \left(\phi(f_i(b)) - \psi(f_i(b))\right) \right| \\
&\leq \sum_{i=1}^k \pi_i \left|\phi(f_i(b)) - \psi(f_i(b))\right| \\
&\leq \sum_{i=1}^k \pi_i L(f_i(b))d_L(\phi, \psi) \\
&\leq \sum_{i=1}^k \pi_i \dilup(f_i) d_L(\phi, \psi) \\
&= \dilupsum(\pi\cdot\mathbb{F}^*) d_L(\phi, \psi),
\end{align*}
where the second inequality follows from the equivalent formulation of $d_L$ mentioned in Section \ref{secMetricSpaces} and the fact that the distance between states in $\States$ can be calculated by pairing against elements of $\samA$ rather than just $\saA$. Taking the supremum over $b \in \B_1$ gives the result.
\end{proof}

The previous theorem and lemma provide the immediate corollary:

\begin{thm}\label{thmExistenceSelfSimilarState}
Let $(\A, L)$ be an extended complete spectral metric space, $\mathbb{F} = (f_1, \ldots, f_k)$ a dual IFS on $\A$ and $\pi \in \simp{k}$. Suppose $\dilupsum(\pi\cdot\mathbb{F}^*)<1$, and that there exists a $\phi_0 \in \States$ such that
\[d_L\left(\phi_0, (\pi\cdot\mathbb{F}^*)(\phi_0)\right)<\infty.\]
Then there exists at least one $(\pi\cdot\mathbb{F}^*)$-self-similar state. In particular, the sequence
\[\left((\pi\cdot\mathbb{F}^*)^M(\phi_0)\right)_{M=0}^\infty \subset \States\]
is $d_L$-convergent with $(\pi\cdot\mathbb{F}^*)$-self-similar limit.
The set of all $(\pi\cdot\mathbb{F}^*)$-self-similar states is convex, and such that if $\phi$ and $\phi'$ are self-similar and $d_L(\phi, \phi')<\infty$, then $\phi = \phi'$.
\end{thm}

\begin{proof}
Apply the Banach fixed point theorem, with the complication that we are working in an extended complete metric space $(\States, d_L)$. Convexity of the set of fixed states follows from linearity of $\pi\cdot\mathbb{F}^*$
\end{proof}

\section{The fractal: support projections}

Recall that the support of a self-similar measure related to a classical (strictly contractive) IFS is exactly the attractor of the IFS. In this section, we will study the support projections (and closed support projections) of self-similar states related to a dual IFS, and see that these objects are self-similar and intrinsic. 

Given any relatively proper $*$-homomorphism $f:A \to M(B)$, note that the universal property of the enveloping von Neumann algebra provides us a $\sigma$-weakly continuous extension $f:\evnA \to B''$ (which we denote by the same symbol), acquired through viewing $M(B)$ as acting on the universal representation of $B$, in which the bicommutant of $M(B)$ is $B''$.

The $\sigma$-weakly continuous extension $f:\evnA \to B''$ also extends the extension $f:\mA \to M(B)$ guaranteed by Theorem \ref{thmLanceMorphism}. This follows by density of $A \subset M(A), \evnA$ in their relevant topologies, combined with the fact that if a net $(a_n)_{n \in N} \subset \mA$ converges strictly then it converges to the same limit $\sigma$-weakly (since the strictly continuous states of $M(A)$ are exactly the normal states of $\evnA$).

We denote by $\Proj$ the set of projections in $\evnA$. We are interested in analogues of the sets of open and closed projections in $\Proj$. Open and closed projections have been introduced by Akemann (see \cite{Akemann1969} for the unital setting).

\begin{defn}[\cite{Akemann1969}, Definition II.1]
A projection $p \in \Proj$ is called open if there exists a net $(a_n)_{n \in N} \subset \A$ with $0 \leq a_n \nearrow p$ in the $\sigma$-weak topology in $\evnA$. A projection $p$ is called closed if $1-p$ is open. The set of all open projections is denoted $\oProj$ and the set of all closed projections is denoted $\cProj$.
\end{defn}

Akemann shows (\cite{Akemann1969}, Prop II.5, Example II.6) that the suprema (in $\Proj$) of any set of open projections is again open, but infima of finite sets of open projections may not be open (dually, $\cProj \subset \Proj$ is closed under arbitrary meets but not necessarily under finite joins). Given any $p \in \Proj$ we may thus define its interior
\[\mathring{p} = \bigvee \{q \in \oProj; q \leq p\}\]
which is the largest open projection less than $p$, and its closure
\[\overline{p} = \bigwedge \{q \in \cProj; q \geq p\},\]
which is the smallest closed projection larger than $p$.
Note that a projection is open (resp. closed) if and only if it is equal to its interior (resp. closure).

We have the following useful characterisation of closed projections courtesy of Alfsen and Shultz.

\begin{lemma}[preceding Definition 3.56 of \cite{AlfsenShultz2001}]\label{lemmaClosedUpperSemicontinuous}
$p \in \Proj$ is closed (resp. open) if and only if the pairing
\[\States \ni \psi \mapsto \psi(p) \in [0,1]\]
is weak-$*$ upper (resp. lower) semicontinuous.
\end{lemma}

Strictly speaking, Alfsen and Shultz provide the above in the unital setting, but their reasoning applies also to the non-unital setting (by working on the weak-$*$ compact set of quasi-states rather than the non-compact state space).

It is immediate that the image of an open or closed projection under a proper $*$-homomorphism is again open or closed (which is the noncommutative analogue of the fact that preimages of open or closed sets under continuous maps are again open or closed). We verify that this holds for relatively proper $*$-homomorphisms too.

\begin{lemma}\label{lemmaRelProperClosed}
Suppose $f:A \to M(B)$ is relatively proper. Then 
\[f\left(\oProj\right) \subset \mathring{\mathcal{P}}\left(B\right) \text{ and } f\left(\cProj\right) \subset \overline{\mathcal{P}}\left(B\right).\]
\end{lemma}

\begin{proof}
We only have to verify one of the two inclusion, the other follows by taking orthogonal complements and using the fact that the extension of $f$ is unital (either viewed $\mA \to M(B)$ or $\evnA \to B''$).

Suppose $p \in \oProj$. Pick $\left(a_n\right)_{n \in N} \in \A$ with $0 \leq a_n \nearrow p$ ($\sigma$-weakly). Then by $\sigma$-weak continuity of (the relevant extension of) $f$, $0 \leq f(a_n) \nearrow f(p)$.

We have $f(a_n) \in \samA$ for each $n \in N$. Thus by \cite{PedersenCstarAndAutomorphisms1979}, Theorem 3.12.9 and Proposition 3.11.8, the pairing 
\[\mathcal{S}(B) \ni \phi \to \phi(f(a_n))\]
is weak-$*$ continuous. But then the pairing with $f(p)$ is the pointwise supremum of those above, and is hence weak-$*$ lower semicontinuous. Hence $f(p) \in \mathring{\mathcal{P}}\left(B\right)$.
\end{proof}

Recall that for a given state $\phi \in \States$, the support projection $\supp(\phi)$ of $\phi$ is the infimum of projections $p \in \Proj$ such that $\phi(p)=1$. I.e.
\[\supp(\phi) = \bigwedge \{p \in \Proj; ~ \phi(p)=1\}.\]

We will also be interested in the closed support projection $\csupp(\phi)$, which is simply the closure of $\supp(\phi)$. From the discussion just prior, we alternatively have
\[\csupp(\phi) = \bigwedge \{p \in \cProj; ~ \phi(p)=1\}.\]

It should be noted that $\phi(\supp(\phi)) = \phi(\csupp(\phi)) = 1$, despite the fact that a similar statement for Borel probability measures does not always hold (our (normal) states correspond to noncommutative analogues of \emph{regular} Borel measures, for which the classical support always has full measure).

\begin{thm}\label{thmSupport}
Suppose $(\A,L)$ is an extended complete spectral metric space and that $\mathbb{F} = (f_1, \ldots, f_k)$ is a dual IFS on $\A$.  Suppose $\pi, \pi' \in \osimp{k}$ are such that $\dilupsum(\pi\cdot\mathbb{F}^*), \dilupsum(\pi'\cdot\mathbb{F}^*) <1$. Then for any two states $\phi, \phi'$ which are $(\pi\cdot\mathbb{F}^*)$- and $(\pi'\cdot\mathbb{F}^*)$-self-similar respectively, and are also such that $d_L(\phi, \phi')<\infty$, we have
\[\csupp(\phi) = \csupp(\phi').\]
\end{thm}

Note that we make the restriction that $\pi$ and $\pi'$ belong to the strict standard $(k-1)$-simplex above. This is so that all the relatively proper $*$-homomorphisms making up $\mathbb{F}$ actually show up in both $\pi\cdot\mathbb{F}^*$ and $\pi'\cdot\mathbb{F}^*$. One may alternatively require that $\pi$ and $\pi'$ are non-zero in exactly the same indices.

\begin{proof}
Let $p = \csupp(\phi)$ and $q = \csupp(\phi')$.

To start, note that since $\phi$ is a fixed point of $\pi\cdot\mathbb{F}^*$, we have for any $M$
\begin{align*}
1 &= \phi(p)\\
&= (\pi\cdot\mathbb{F}^*)^M(\phi)(p) \\
&= \sum_{\omega \in \Code{k}{M}} \pi_\omega f_\omega^*\phi(p).
\end{align*}
Since $\sum_{\omega \in \Code{k}{M}} \pi_\omega=1$ and everything is non-negative, this is only possible if $f_\omega^*\phi(p)=1$ for each $\omega \in \Code{k}{M} $.

Using the fact that $\phi'$ is fixed by $\pi'\cdot\mathbb{F}^*$ combined with Lemma \ref{lemmaLip}, for $M \in \NN$ we have
\[d_L\left(\phi', \left(\pi'\cdot\mathbb{F}^*\right)^M(\phi)\right) \leq \dilupsum(\pi'\cdot\mathbb{F}^*)^M d_L(\phi', \phi).\]
Since $d_L(\phi', \phi)<\infty$, we may thus conclude $\left(\pi'\cdot\mathbb{F}^*\right)^M\phi \dLto \phi'$. As $d_L$-convergence implies weak-$*$ convergence, we also have $\left(\pi'\cdot\mathbb{F}^*\right)^M\phi \wsto \phi'$.

We claim that for every $M \in \NN$, $\left(\pi'\cdot\mathbb{F}^*\right)^M\phi(p)=1$. To see this, we have
\begin{align*}
\left(\pi'\cdot\mathbb{F}^*\right)^M\phi(p) &= \sum_{\omega \in \Code{k}{M}} \pi'_\omega f_\omega^*\phi(p) \\
&= \sum_{\omega \in \Code{k}{M}} \pi'_\omega \\
&=1.
\end{align*}

Since $p$ is a closed projection, pairing with $p$ is weak-$*$ upper semicontinuous on $\States$ (Lemma \ref{lemmaClosedUpperSemicontinuous}). Thus
\[1 = \limsup_{M \to \infty} \left(\pi'\cdot\mathbb{F}^*\right)^M(\phi)(p) \leq \phi'(p) \leq 1,\]
and thus $\phi'(p)=1$. By minimality of $q$ among closed projections with this property, we deduce $q \leq p$. By swapping the roles of $p$ and $q$, one also concludes that $p \leq q$ and so they must be equal.
\end{proof}

In light of Theorem \ref{thmSupport}, we may consider the closed support projection corresponding to a dual IFS $\mathbb{F}$ and $\pi \in \osimp{k}$ as providing an analogue of the attractor of $\mathbb{F}$, without dependence of the specific choice of weights $\pi \in \osimp{k}$ (up to complications arising from having an extended metric).

One would hope that the (closed) support projection of a self-similar state is itself self-similar. However, there is a complication: our morphisms go the wrong way around. A dual IFS on a commutative $C^*$-algebra corresponds to pre-composition by a classical IFS on the spectrum, and so the action on projections (corresponding to subsets of the spectrum) is to take pre-images. However, the attractor of a classical IFS is self-similar under direct images.

One might hope that in the noncommutative setting we might be able to get away with taking direct images at the level of the state space (under the dual maps $\mathbb{F}^* = \left(f_1^*, \ldots, f_k^*\right)$), and then use some equivalence between (certain) subsets of the state space and (certain) projections in $\evnA$. However, the relevant subsets of the state space are (weak-$*$ closed) faces, and it is not the case that the image of a face is necessarily a face.

To resolve the issue, we have the following definition.

\begin{defn}
Given a relatively proper morphism $f:A \to M(B)$ between $C^*$-algebras, we define the upper pre-image $\uppreimage{f} : \mathcal{P}(B'') \to \mathcal{P}(A'')$ by
\[\uppreimage{f}(p) = \bigwedge \{q \in \mathcal{P}(A); ~ f(q) \geq p\}.\]
%and
%\[\lowpreimage{f}(p) = \bigvee \{q \in \mathcal{P}(A); ~ f(q) \leq p\}.\]
\end{defn}

Since we only consider relatively proper morphisms (which have unital $\sigma$-weakly continuous extensions), the meet above is never over an empty set. Note that if (the $\sigma$-weakly continuous extension of) $f$ is invertible, then $\uppreimage{f} = f^{-1}|_{\mathcal{P}(B'')}$.

With this definition in hand, we can describe the self-similar nature of the support projections of a self-similar state. Note that the below theorem applies even in the absence of contractivity (i.e. regardless of the value of $\dilupsum(\pi\cdot\mathbb{F}^*)$).

\begin{thm}\label{thmSelfSimilar}
Let $\A$ be $C^*$-algebra, $\mathbb{F} = (f_1, \ldots, f_k)$ a dual IFS on $\A$ and $\pi \in \osimp{k}$. Suppose that $\phi \in \States$ is $(\pi\cdot\mathbb{F}^*)$-self-similar. Then 
\[ \supp(\phi) = \bigvee_{i=1}^k \uppreimage{(f_i)}(\supp(\phi)),\]
and
\[ \csupp(\phi) = \overline{\left(\bigvee_{i=1}^k \uppreimage{(f_i)}(\csupp(\phi))\right)}.\]
\end{thm}

\begin{proof}
We will verify the first statement by showing inequalities hold in both directions.

It was shown in the proof of Theorem \ref{thmSupport} that for each $i$ we have $\phi(f_i(\supp(\phi))) = 1$, which implies that $f_i(\supp(\phi)) \geq \supp(\phi)$ for each $i$ by minimality. So by definition of the upper pre-image,
\[\uppreimage{(f_i)}(\supp(\phi)) \leq \supp(\phi)\]
for each $i$, and hence
\[ \bigvee_{i=1}^k \uppreimage{(f_i)}(\supp(\phi)) \leq \supp(\phi).\]

For the other inequality, we have
\begin{align*}
\phi\left(\bigvee_{i=1}^k \uppreimage{(f_i)}(\supp(\phi)) \right) &= (\pi\cdot\mathbb{F}^*)(\phi)\left(\bigvee_{i=1}^k \uppreimage{(f_i)}(\supp(\phi)) \right) \\
&= \sum_{j=1}^k \pi_j \phi\circ f_j \left(\bigvee_{i=1}^k \uppreimage{(f_i)}(\supp(\phi)) \right) \\
&\geq \sum_{j=1}^k \pi_j \phi\circ f_j \left(\uppreimage{(f_i)}(\supp(\phi))\right) \\
&\geq \sum_{j=1}^k \pi_j \phi(\supp(\phi)) \\
&= \sum_{j=1}^k \pi_j \\
&= 1,
\end{align*}
where we have used the fact that each $f_i$ is order preserving, and $f_i\left(\uppreimage{(f_i)}(p)\right)\geq p$ by definition of the upper pre-image. Thus by minimality of $\supp(\phi)$, we have the other inequality
\[\supp(\phi) \leq \bigvee_{i=1}^k \uppreimage{f_i}(\supp(\phi)).\]

Having shown inequalities in both direction, we conclude that
\[ \supp(\phi) = \bigvee_{i=1}^k \uppreimage{(f_i)}(\supp(\phi)).\]

We now prove the second statement. Taking the closure of both sides above yields
\[\csupp(\phi) = \overline{\left(\bigvee_{i=1}^k \uppreimage{(f_i)}(\supp(\phi))\right)}.\]

Just as before, since $\phi(\csupp(\phi)) = 1$ we necessarily have that $\phi(f_i(\csupp(\phi))) = 1$ for each $i$. Each $f_i(\csupp(\phi))$ is closed by Lemma \ref{lemmaRelProperClosed}, so minimality of $\csupp(\phi)$ implies that
\[\csupp(\phi) \leq f_i(\csupp(\phi)).\]
so by the same reasoning as previously,
 \[ \bigvee_{i=1}^k \uppreimage{(f_i)}(\csupp(\phi)) \leq \csupp(\phi).\]
 Since the projection on the right is closed and closure preserves order, we have
\[ \overline{\left(\bigvee_{i=1}^k \uppreimage{(f_i)}(\csupp(\phi))\right)} \leq \csupp(\phi).\]

Combined with the fact that
\[\overline{\left(\bigvee_{i=1}^k \uppreimage{(f_i)}(\supp(\phi))\right)} \leq \overline{\left(\bigvee_{i=1}^k \uppreimage{(f_i)}(\csupp(\phi))\right)},\]
we find by squeezing that
\[ \csupp(\phi) = \overline{\left(\bigvee_{i=1}^k \uppreimage{(f_i)}(\csupp(\phi))\right)}.\]

\end{proof}

Thus the support projection and closed support projection of a self-similar state are both themselves self-similar in the lattice of projections and closed projections respectively.

The previous theorem may cause some issues. Note that it is only the closed support projections to which Theorem \ref{thmSupport} applies, so it is the closed projections which capture intrinsic information about the iterated function system. However, the closed support projection is only self-similar up to taking the closure of the join of a collection of projections, and the closure of a join is not very well understood (as mentioned previously, the join of a finite set of closed projections may not even be closed). On the other hand, the support projection is self-similar on the nose, but has the issue that it may depend on the specific choice of weights $\pi \in \osimp{k}$.

\section{Representation formulas related to strictly contractive dual IFS}\label{secStrictlyContractive}

We move our attention to strictly contractive dual IFS, i.e. those for which
\[\dilupsup(\mathbb{F}) = \max_{i=1, \ldots, n} \dilup(f_i) \in [0, \infty] <1.\]
In the classical commutative setting, strictly contractive IFS are known to have compact attractors, and the attractor has an explicit formula in terms of fixed points of compositions of the contractions. In this section we show that a similar thing holds even when forgoing commutativity.

Note $\sup_{\pi \in \osimp{k}}\dilupsum{\pi\mathbb{F}} = \dilupsup{\mathbb{F}}$, so the results of the previous sections can be applied to obtain existence and uniqueness of $(\pi \cdot \mathbb{F}^*)$-self-similar states for any $\pi \in \osimp{k}$ for any strictly contractive dual IFS $\mathbb{F}$ (up to issues with infinite distances).

We introduce some notation before the next few theorems.

We let $\infCode{k}$ be the set of all sequences of elements of $\{0, \ldots, k\}$, which we equip with the product topology (with each factor given the discrete topology) and the $\sigma$-algebra given by the tensor product $\sigma$-algebra of the power set in each factor (i.e. generated by sets of sequences with prescribed values in finitely many indices). Note that $\infCode{k}$ is compact in the given topology, by Tychonoff's theorem.

Given $\pi \in \osimp{k}$, we let $\mathbb{P}_\pi$ be the product probability measure on $\infCode{k}$ such that the pushforward onto each factor $\{0, \ldots, k\}$ has density $\pi$. Note that this implies
\[\pi_\omega = \mathbb{P}_\pi(\{\omega' \in \infCode{k}; ~ \omega'|_M = \omega\})\]
for each $\omega \in \Code{k}{M}$.

We identify each $\omega \in \Code{k}{M}$ with $\overline{\omega} = (\omega, \omega, \ldots) \in \infCode{k}$, and thus think of $\Code{k}{M}$ as a subset of $\infCode{k}$. Similarly, for $\omega \in \infCode{k}$, we set $\omega|_M = (\omega_1, \ldots, \omega_M) \in \Code{k}{M}$.

For the rest of this section, we will suppose $(\A, L)$ is an extended complete spectral metric space, $\mathbb{F}$ is a strictly contractive dual IFS on $(\A,L)$ with $\dilupsup{\mathbb{F}} =: \Lambda <1$, and $\phi_0 \in \States$ is such that
\[C := \max_{i=1, \ldots, k} d_L(\phi_0, f_i(\phi_0))< \infty.\]

The Banach fixed point theorem implies that for each $\omega \in \infCode{k}$, the sequence $\left(f^*_{\omega|_M}(\phi_0)\right)_{M=1}^\infty$ is $d_L$-convergent, whose limit we denote $\phi_\omega$. 
Similarly, for each $\omega \in \Code{k}{M}$, $f^*_\omega$ has a unique fixed state within a finite distance of $\phi_0$ given by $\phi_{\overline{\omega}}$. We will simplify notation and denote $\phi_{\overline{\omega}}$ by $\phi_\omega$ also.

Through summing the bounds implied by the Banach fixed point theorem, one may easily deduce that $d_L(\phi_0, \phi_\omega)< C/(1-\Lambda)$ for each $\omega \in \infCode{k}$.

It is known that the map $\pi \mapsto \phi_\pi$ is continuous with respect to the product topology on $\infCode{k}$ and $d_L$, see Theorem 3.1(3)(vii) of \cite{Hutchinson}. Hence $\{\phi_\omega; ~\omega \in \infCode{k}\}$ is $d_L$-compact.

It should be noted that the set $K_\mathbb{F} := \{\phi_\omega; ~\omega \in \infCode{k}\}$ is exactly ``the" $d_L$-compact attractor of $\mathbb{F}^*$ when viewed as a classical contractive IFS on the classical extended complete metric space $(\States, d_L)$ (which is unique among $d_L$-compact sets within a finite distance of $\phi_0$). We will discuss compactness in the setting of noncommutative topology in a later section.

\begin{thm}\label{thmContractiveCharacterisation}
For any $\pi \in \osimp{k}$, the pair of sequences $(\phi^M)_{M=1}^\infty, (\phi_M)_{M=1}^\infty \subset \States$, defined by
\[\phi^M := \sum_{\omega \in \Code{k}{M}} \pi_\omega \phi_\omega\]
and
\[\phi_M := (\pi \cdot \mathbb{F}^*)^M(\phi_0)\]
are both $d_L$-convergent, and have the same $d_L$-limit $\phi_\pi$. In particular, $\phi_\pi$ is the unique $(\pi \cdot \mathbb{F}^*)$-self-similar state within a finite distance of $\phi_0$.
\end{thm}

\begin{proof}
Convexity of $d_L$ provides
\[d_L\left(\phi_0, (\pi \cdot \mathbb{F}^*)(\phi_0)\right) \leq \sum_{i=1}^k \pi_i d_L\left(\phi_0, f^*_i(\phi_0)\right) \leq C,\]
so we may apply Theorem \ref{thmExistenceSelfSimilarState} to deduce that $\phi_M$ is $d_L$-convergent to a $(\pi \cdot \mathbb{F}^*)$-self-similar state.

Fix $b \in \samA$. Then
\begin{align*}
|\phi^M(b)-\phi_M(b)| &= \left|\sum_{\omega \in \Code{k}{M}} \pi_\omega \phi_\omega(b) - (\pi \cdot \mathbb{F}^*)^M(\phi_0)(b)\right| \\
&= \left|\sum_{\omega \in \Code{k}{M}} \pi_\omega \phi_\omega(b) - \pi_\omega f^*_\omega(\phi_0)(b)\right| \\
&\leq \sum_{\omega \in \Code{k}{M}} \pi_\omega \left|\phi_\omega(b) - f^*_\omega(\phi_0)(b)\right| \\
&= \sum_{\omega \in \Code{k}{M}} \pi_\omega \left|f^*_\omega(\phi_\omega)(b) - f^*_\omega(\phi_0)(b)\right| \\
&= \sum_{\omega \in \Code{k}{M}} \pi_\omega \left|\phi_\omega(f_\omega(b)) - \phi_0(f_\omega(b))\right| \\
&\leq \sum_{\omega \in \Code{k}{M}} \pi_\omega d_L\left(\phi_\omega, \phi_0\right)L(f_\omega(b)) \\
&\leq \sum_{\omega \in \Code{k}{M}} \pi_\omega \frac{C}{1-\Lambda} \Lambda^M L(b) \\
&= \frac{CL(b)}{1-\Lambda} \Lambda^M,
\end{align*}
where the third equality holds by definition of $\phi_\omega$, so $d^L(\phi_M, \phi^M) \leq C\Lambda^M/(1-\Lambda)  \to 0$ as $M \to \infty$. As already noted, $\phi_M$ $d_L$-converges to the $(\pi\cdot\mathbb{F}^*)$-self-similar state $\phi_\pi$, so $\phi^M$ must $d_L$-converge to $\phi_\pi$ also.

\end{proof}

The above theorem provides the following two corollaries, which are clear noncommutative analogues of what holds in the classical setting (Theorems 3.1(3)(v) and 4.4(4)(i) of \cite{Hutchinson}). Namely, that the attractor of a classical strictly contractive IFS is the closure of the set of fixed points of compositions of the maps in said IFS, and that the associated self-similar measures are given explicitly by a pullback via the code map.

\begin{cor}\label{corDenseSubprojection}
\[\csupp\left(\phi_\pi\right) = \overline{\left(\bigvee_{\omega \in \Code{k}{M}, M \in \mathbb{N}} \supp\left(\phi_\omega\right)\right)}.\]
\end{cor}

\begin{proof}
Let $p$ and $q$ denote the projections on the left and right of the equality above, respectively. 

We first show $p \leq q$. Note that $\phi^M(q)=1$ for all $M$, by definition. Combining this with the facts that $\phi^M \dLto \phi_\pi$ (and hence the same limit holds weak-$*$), and that the pairing with $q$ is weak-$*$ upper semicontinuous (by Lemma \ref{lemmaClosedUpperSemicontinuous}), we have
\[1 \geq \phi_\pi(q) \geq \limsup_{M \to \infty} \phi^M(q) = 1.\]
Hence $\supp(\phi_\pi) \leq q$, and by closedness of $q$ again we conclude $p = \csupp(\phi_\pi) \leq q$.

To finish, we show $p \geq q$. By self-similarity, we find
\[(\pi\cdot\mathbb{F}^*)^M\phi_\pi(p) = \sum_{\omega \in \Code{k}{M}}\pi_\omega f_\omega^*\phi_\pi(p) = 1\]
for each $M$. By (strict) positivity of each $\pi_\omega$, this implies that $f_\omega^*\phi_\pi(p) = 1$ for each $\omega \in \Code{k}{M}$ and each $M$.

Fixing some such $\omega \in \Code{k}{M}$ and noting that $\omega^j = (\omega, \ldots, \omega) \in \Code{k}{jM}$ for each $j$, this gives
\[1 = f^*_{\omega^j}\phi_\pi(p) = \left(f_\omega^*\right)^j\phi_\pi(p).\]
However, since $d_L(\phi_\omega, \phi_\pi) <\infty$ we know that $\left(f_\omega^*\right)^j\phi_\pi \dLto \phi_\omega$, and hence the same convergence holds in the weak-$*$ topology (Lemma \ref{lemmadLtoImplieswsto}). So Lemma \ref{lemmaClosedUpperSemicontinuous} implies
\[1 = \limsup_{j \to \infty} \left(f_\omega^*\right)^j\phi_\pi(p) \leq \phi_\omega(p)\leq 1\]
and hence $\phi_\omega(p) = 1$. In other words, $\supp\left(\phi_\omega\right) \leq p$ for each $\omega \in \Code{k}{M}$ and each $M$, from which $p \geq q$ follows from closedness of $p$.
\end{proof}

\begin{cor}\label{corPushforward}
The map $g:\infCode{k} \to \States$ defined by $g(\omega) = \phi_\omega$
is weak-$*$ measurable, and
\[\phi_\pi = \int_{\infCode{k}} \phi_\omega d\mathbb{P}_\pi(\omega),\]
with the integral converging in the weak-$*$ sense.
In particular, $\phi_\pi$ is contained in the weak-$*$ closed convex hull of $K_\mathbb{F}$.
\end{cor}

\begin{proof}
Consider the sequence of functions $g_M:\infCode{k} \to \States$ defined by $g_M(\omega) = \phi_{\omega|_M}$. Each $g_M$ is simple (i.e. a finite $A^*$-weighted sum of indicator functions of measurable sets), and hence is weak-$*$ measurable.

Fix $a \in \A$. Then $g_M(\omega)(a) = \phi_{\omega|_M}(a) \to \phi_\omega(a)$ since $\phi_{\omega|_M} \dLto \phi_\omega$ (so convergence also holds weak-$*$, by Lemma \ref{lemmadLtoImplieswsto}). So $g(\cdot)(a)$ is the pointwise limit of the measurable functions $g_M(\cdot)(a)$, and so is itself measurable.

Noting that $|g_M(\cdot)(a)|, |g(\cdot)(a)| \leq ||a|| \in L^1(\infCode, \mathbb{P}_\pi)$, we may thus apply the dominated convergence theorem to deduce
\begin{align*}
\int_{\infCode{k}} \phi_\omega(a) d\mathbb{P}_\pi(\omega) &= \int_{\infCode{k}} g(\omega)(a)d\mathbb{P}_\pi(\omega) \\
&= \lim_{M \to \infty} \int_{\infCode{k}} g_M(\omega)(a)d\mathbb{P}_\pi(\omega) \\
&= \lim_{M \to \infty} \sum_{\omega' \in {\Code{k}{M}}} \phi_{\omega'}(a)\mathbb{P}_\pi(\{\omega \in \infCode{k}; ~ \omega|_M = \omega'\}) \\
&= \lim_{M \to \infty} \sum_{\omega' \in {\Code{k}{M}}} \pi_{\omega'}\phi_{\omega'}(a) \\
&= \lim_{M \to \infty} \phi^M(a),
\end{align*}
where the equality $\pi_{\omega'} = \mathbb{P}_\pi(\{\omega \in \infCode{k}; ~ \omega|_M = \omega'\})$ was noted in the introduction to this section. From this point, uniqueness of weak-$*$ limits implies the second statement. The last statement follows immediately from noting that the above approximation is achieved through convex combination of states in $K_\mathbb{F}$.
\end{proof}

%%%%%%%%%%%%%%%%%%%%%%%%%%%%%%%%%%%%%%%%%%%%%%%%%%%%%%%%%%%%%%%%%%%%%%%%%%%%%%%%%%%%%%%%%%%%%%%%%%%%%%%%%%%%%%%%%%%%%%%%%%%%%%%%%%%%%%%%%%%%%%%%%%%%%%%%%%%%%%%%%%%%%%%%%%%%%%%%%%%%%%%%%%%%%%%%%%%%%%%%%%%%%%%%%%%%%%%%%%%%%%%%%%%%%%%%%%%%%%%%%%%%%%%%%%%%%%%%%%%%%%%%%%%%%%%%%%%%%%%%%%%%%%%%%%%%%%%%%%%%%%%%%%%%%%%%%%%%%%%%%%%%%%%%%%%%%%%%%%%%%%%%%%%%%%%%%%%%%%%%%%%%%%%%%%%%%%%%%%%%%%%%%%%%%%%%%%%%%%%%%%%%%%%%%%%%%%%%%%%%%%%%%%%%%%%%%%%%%%%%%%%%%%%%%%%%%%%%%%%%%%%%%%%%%%%%%%%%%%%%%%%%%%%%%%%%%%%%%%%%%%%%%%%%%%%%%%%%%%%%%%%%%%%%%%%%%%%%%%%%%%%%%%%%%%%%%%%%%%%%%%%%%%%%%%%%%%%%%%%%%%%%%%%%%%%%%%%%%%%%%%%%%%%%%%%%%%%%%%%%%%%%%%%%%%%%%%%%%%%%%%%%%%%%%%%%%%%%%%%%%%%%%%%%%%%%%%%%%%%%%%%%%%%%%%%%%%%%

\section{Closing Comments and Future Work}

There is much left to be done regarding self-similar states and fractal projections. We illustrate some of these topics and progress towards their understanding below.

\subsection{The topology induced by $d_L$}\label{subsecTopology}\

In the classical (non-extended) setting, it is known that any strictly contractive IFS has compactly supported self-similar states \cite{Hutchinson}. This result heavily relies on the fact that the topology induced by the Monge-Kantorovich distance is related to the weak-$*$ topology on the state space (although they disagree on unbounded domains).

The question of whether the topology induced by $d_L$ agrees with the weak-$*$ topology in the unital setting has a conclusive answer provided by Rieffel's \cite{Rieffel1998} Theorem 1.8. This states that if $\A$ is unital and $(\A,L)$ is an extended complete spectral metric space, then $\States$ is $d_L$-bounded if and only if $\B_1 /\mathbb{R}I$ is bounded, and $d_L$ induces the weak-$*$ topology on $\States$ if and only if $\B_1 /\mathbb{R}I$ is totally bounded (bounded/totally bounded with respect to the quotient norm).

The non-unital setting is much more nuanced. Indeed, as already mentioned, even in the commutative case one does not actually expect the $d_L$ topology and the weak-$*$ topology to agree. Latr{\'e}moli{\`e}re provides an example of a state $\phi$ on $C_0(\RR)$ which is infinitely distant (with respect to the usual Monge-Kantorovich distance) from the Dirac mass at $0$ (\cite{Latremoliere2007}, Example 1.2). In fact, $\phi$ is infinitely distant from any compactly supported state (as can be verified by slightly modifying Latr{\'e}moli{\`e}re's argument). Compactly supported states are clearly weak-$*$ dense in the commutative $\sigma$-compact setting (verified in general in the upcoming Lemma \ref{lemmacStatesDense}), so the $d_L$ topology does not agree with the weak-$*$ topology in this case.

To circumvent the above issues (mainly present through the fact that $d_L$ may take the value $\infty$), both localisation and artificial bounding of $d_L$ have been trialled (see for example \cite{Latremoliere2007}, \cite{Latremoliere2013} and references therein).

However, work of Cagnache, D'Andrea, Martinetti, and Wallet (see \cite{Cagnache2011}) on the metric structure of the Moyal plane suggests that the infinite values of $d_L$ may be a feature rather than a bug. Cagnache et. al. provide equivalent states of the Moyal plane $C^*$-algebra which are infinitely distant with respect to the spectral distance induced by a very reasonable Lipschitz seminorm (arising from a very reasonable spectral triple). Importantly, the Lipschitz seminorm considered is such that $L(a)$ is non-zero for every $a$ which isn't a multiple of the unit (otherwise infinitely distant states are easy to find). Given that the constructed states are equivalent, Kadison's transitivity theorem implies that they both belong to the same weak-$*$ compact face of state space, so one should not think of the infinite distance as arising from a localisation issue as in the example of Latr{\'e}moli{\`e}re.

With all this in mind, we throw our own tentative topologically-inclined definition into the ring (after some preliminaries).

\begin{defn}\label{defncA}
Let $\A$ be a $C^*$-algebra. An $a \in \saA$ is called a \emph{compactly supported bump} if $0 \leq a \leq I$, $||a||=1$, and there exists $I_a \in \saA$ with $0 \leq I_a \leq I$ and
\[I_a a = a.\]
The set of compactly supported bumps is denoted $\cA$.
\end{defn}

\begin{lemma}\label{lemmaCompactsDense}
$\cA \subset \A$ has dense span.
\end{lemma}

\begin{proof}
%Suppose $a, a' \in \cA$ with $l, r, l', r' \in \A$ as in Definition \ref{defncA}. Then
%\[laa' = aa'r' = aa',\]
%and
%\[a^*l = (la)^* = a^* = (ar)^* = ra^*.\]
%So $aa', a^* \in \cA$.

Given any $c \in \A^+$ with $||c||=1$, and $n \in \NN$, let $g_n:[0,1] \to [0,1]$ be the function which is $0$ on $[0,1/n]$, agrees with the identity on $[2/n,1]$ and linearly interpolates between, and let $I_n:[0, 1] \to [0,1]$ be identically $1$ on $[1/n, 1]$ and $0$ at $0$, linearly interpolating between. Then by functional calculus for non-unital $C^*$-algebras, $I_n(c) \in \saA$, $0 \leq I_n(c) \leq I$, and $I_n(c)g_n(c) = g_n(c)$, so $g_n(C) \in \cA$ and $||g_n(c) - c|| \leq 2/n$. Since the span of elements in $\A^+$ with norm $1$ is dense in $\A$, the span of $\cA$ is dense in $\A$.
\end{proof}

The above has the following useful consequence.

\begin{lemma}\label{lemmaStrictOnBounded}
The restriction of the strict topology to bounded subsets $\mA$ is induced by the seminorms $b \mapsto ||ab||, ||ba||$ for $a \in \cA$.
\end{lemma}

\begin{proof}
The set of semi-norms above is a (strict) subset of those defining the strict topology, so one direction is immediate. For the other, suppose $\left(b_n\right)_{n \in N} \subset \mA$ is a bounded net and $b \in \mA$ are such that for any $a \in \cA$, $||a(b_n-b)||, ||(b_n-b)a|| \to 0$. Then the same holds for the span of $\cA$ in $\A$ by the triangle inequality.

Now suppose that $a' \in \A$, and fix $\epsilon>0$. Pick a some $a \in \text{span}(\cA)$ with $||a-a'|| \leq \epsilon/4\sup_{n \in N}||b_n||$, and pick $n \in N$ such that for all $m \geq n$, $||a(b_n-b)||, ~||(b_n-b)a||\leq \epsilon/2$. Then
\[||a'(b_n-b)|| \leq ||a'-a|| ||b_n-b|| + ||a(b_n-b)|| \leq \epsilon,\]
with a similar result for the multiplication in the other order, so we are done.
\end{proof}

\begin{defn}
Given $a \in \cA$, we define the $a$-locally flat elements by
\[\locflat{a} := \bigcap_{R>0}a \mB_{R} a.\] 
\end{defn}

The classical picture one should have in mind is that $\locflat{a}$ is the collection of Lipschitz functions whose Lipschitz constant is zero when restricted to the support of $a$.

\begin{lemma}
For any $a \in \cA$, $\locflat{a}$ is a real vector subspace of $\saA$. It is the largest real vector subspace contained in $a\mB_1 a$.
\end{lemma}

\begin{proof}
Suppose $r \in \mathbb{R}$, and $b, b' \in \locflat{a}$. Then for each $R>0$ there exists $b_R, b_R' \in \mB_R$ with $b =ab_Ra, ~b' = ab_R'a$.

Clearly if $r=0$, then $rb \in \locflat{a}$. If $r \neq 0$, then $rb = a\tilde{b}_Ra$ where $\tilde{b}_R = r b_{R/r} \in \mB_R$ for each $R$, so $rb \in \locflat{a}$.

Similarly, $b+b' = a \hat{b}_R a$ where $\hat{b}_R = b_{R/2}+b_{R/2}' \in \mB_R$ for each $R>0$, so $b+b' \in \locflat{a}$.

For maximality: given any $b \in a\mB_1 a$ such that the ray generated by $b$ is also contained in $a\mB_1 a$, one can rescale to deduce that $b \in \locflat{a}$. Since $L$ is a norm, $a\mB_1a$ is convex and invariant under negation, so we conclude that the sum of all rays in $a\mB_1a$ is exactly $\locflat{a}$.
\end{proof}

In a non-extended classical metric space, $\locflat{a}$ is one dimensional and spanned by $a^2$, corresponding to the fact that the only functions with Lipschitz semi-norm $0$ are constant. However, we are working with noncommutative analogues of (potentially) extended metric spaces. If two points $x,y$ in a locally compact extended metric space $(X,d)$ are infinitely distant, then there must be a sequence of arbitrarily flat functions $f_i$ with $|f_i(x)-f_i(y)|$ strictly bounded below. Thus if one were to localise to a compact neighbourhood of $\{x,y\}$ via a compactly supported continuous function $a$, then $\locflat{a}$ will be two dimensional (spanned by a cut-off of the identity, and an appropriate limit of cut-offs of such arbitrarily flat functions).

This brings us to our definition.

\begin{defn}\label{defnBoundedConsistent}
An extended complete spectral metric space $(\A,L)$ is said to be
\begin{itemize}
\item \emph{locally bounded}, if $\locflat{a}$ is one dimensional and $a\mB_1a / \locflat{a}$ is bounded in the quotient norm,
\item \emph{locally semi-bounded}, if $\locflat{a}$ is finite dimensional and $a\mB_1a / \locflat{a}$ is bounded in the quotient norm,
\item \emph{locally consistent}, if $\locflat{a}$ is one dimensional and $a\mB_1a / \locflat{a}$ is totally bounded in the quotient norm,
\item \emph{locally semi-consistent}, if $\locflat{a}$ is finite dimensional and $a\mB_1a / \locflat{a}$ is totally bounded in the quotient norm,
\end{itemize}
for all $a \in \cA \cap \saA$.
\end{defn}

If $A$ is unital (so $I \in \cA$), then the notions of locally bounded (resp. consistent) correspond to those of \cite{Rieffel2003}, and are equivalent to the fact that $d_L$ is bounded (resp. generates the weak-$*$ topology on $\States$). The same proofs can be slightly modified to show analogous localised versions of Rieffel's results in the locally bounded/locally consistent case, while the notions of locally semi-bounded and locally semi-consistent imply similar results providing one restricts to attention to finitely-distant states. Perhaps most importantly, these concepts return what we expect when $\A$ is commutative.

\begin{thm}
If $\A = C_0(X)$ for some locally compact Hausdorff space $X$ and $(A,L)$ is an extended complete spectral metric space, then $(\A,L)$ is
\begin{itemize}
\item locally bounded, if and only if $(X, d_L)$ is a non-extended complete metric space,
\item locally semi-bounded, if and only if $(X, d_L)$ is an extended complete metric space,
\item locally consistent, if and only if $(X, d_L)$ is a non-extended complete metric space and $d_L$ induces the topology of $X$,
\item locally semi-consistent, if and only if $(X, d_L)$ is an extended complete metric space and $d_L$ induces the topology of $X$.
\end{itemize}
\end{thm}

The proof of the first two points above is a straightforward calculation from the definitions. The last two points are a fairly simple application of the Arzela-Ascoli theorem on the compact subsets of $X$.

Note that if $(\A,L)$ is locally semi-bounded, then given any $a \in \cA$ and states $\phi, \psi$ with $\phi(a) = \psi(a) = 1$, then $d_L(\phi, \psi) = \infty$ if and only if $\phi$ and $\psi$ disagree on $\locflat{a}$, otherwise $d_L(\phi, \psi) \leq 2 \text{diam}\left(a\mB_1a / \locflat{a}\right)$. So while the weakest point of Definition \ref{defnBoundedConsistent} allows for jointly localised states to be infinitely distant, this only occurs for a finite dimensional collection of states. In general, we have the following theorem.

\begin{thm}
Let $(\A,L)$ be an extended complete spectral metric space. Then $(\A,L)$ is
\begin{itemize}
\item locally bounded, if and only if $d_L$ is bounded on $F$,
\item locally semi-bounded, if and only if $d_L$ restricted to $F$ takes values in a set of the form $[0, C(F)] \cup \{\infty\}$ for some $C(F) \in (0, \infty)$, and for each $\phi \in F$, those states in $F$ which are infinitely distant from $\phi$ are contained in a finite dimensional convex set,
\item locally consistent, if and only if the restriction of $d_L$ to $F$ induces the restriction of the weak-$*$ topology to $F$,
\item locally semi-consistent, if and only if for each $\phi \in F$, the restriction of $d_L$ to $F_\phi := \{\psi \in F; ~ d_L(\phi, \psi)<\infty\}$ induces the restriction of the weak-$*$ topology on $F_\phi$, and those states in $F$ which are infinitely distant from $\phi$ are contained in a finite dimensional convex set,
\end{itemize}
for each non-empty weak-$*$ compact face $F$ of $\States$.
\end{thm}

We omit the proof, because it is not particularly enlightening and can be found in parts scattered throughout the literature. It is a fairly laborious generalisation of the proofs of Proposition 1.6 and Theorem 1.8 of \cite{Rieffel1998}, or Lemma 2.4 of \cite{Latremoliere2007}, combined with the fact that for any weak-$*$ compact face $F$ of $\States$ there exists some $a \in \cA$ such that $\phi(a) = 1$ for all $\phi \in F$ (which follows from Proposition 3.54 of \cite{AlfsenShultz2001} on the unitisation of $\A$, combined with some simple functional calculus arguments like in the proof of Lemma \ref{lemmaCompactsDense}).

Besides the fact that Definition \ref{defnBoundedConsistent} recaptures the commutative setting, it is also general enough to allow for some infinitely distant states while still restrictive enough to hopefully be useful (see the next subsection).

In particular, we claim that the example of Cagnache et. al. in \cite{Cagnache2011} is locally semi-consistent. The $C^*$-algebra $\A$ considered in \cite{Cagnache2011} is isomorphic to the algebra of compact operators on a separable Hilbert space, for which $\cA$ is exactly the set of positive norm one finite rank elements. Given any $a \in \cA$, we thus find that $a \mB_1 a$ is a finite dimensional convex set. It is then a general result that the quotient of $a \mB_1 a$ by its largest real vector subspace is a compact convex set, so the example of \cite{Cagnache2011} is indeed locally semi-consistent.

The setting of locally semi-consistent extended complete spectral metric spaces will be explored further in future work. Some preliminary investigations have shown such spaces exhibit good compactness properties with respect to the strict topology. In particular, we have the following mean ergodic-style conjecture, the proof of which is mostly finalised and which will be presented in said future work.

\begin{conj}\label{conjMeanErgodic}
Suppose $(\A,L)$ is a locally semi-consistent extended complete spectral metric space, and $f: \A \to \mA$ is a relatively proper $*$-homomorphism such that there exists some $\Lambda \leq 1$ with $L(T(b)) \leq \Lambda L(b)$ for all $b \in \mB$. Then for any $b \in \mB$, the sequence
\[\left(b_n := \frac{1}{n+1}\sum_{i=0}^{n} T^i(b)\right)_{n=0}^\infty\]
converges strictly to some $\hat{b}$ with $T(\hat{b}) = \hat{b}$. If $\Lambda <1$, then $L(\hat{b})=0$.
\end{conj}

\subsection{Compactness}\

We make use of the same notation as in Section \ref{secStrictlyContractive}.

The notion of compactness for closed projections in $\evnA$ was introduced by Akemann in \cite{Akemann1971} and has been studied by many, including recent work of Akemann and Bice \cite{AkemannBice2015}, and Brown \cite{Brown2018}.

\begin{defn}[\cite{Akemann1971}, Definition II.1]
A closed projection $p \in \cProj$ is called compact if there exists $a \in \saA$ with $0 \leq a \leq I$ and $ap = p$.
\end{defn}

There are a variety of other characterisations of compactness throughout the literature, such as weak-$*$ compactness of the face $F(p) := \{\phi \in \States; ~\phi(p) = 1\} \subset \States$ generated by $p$ (see \cite{Akemann1971}). We introduce the following similar definition for states.

\begin{defn}
A state $\phi \in \States$ is called \emph{compactly supported} if it attains its norm on $\A$. I.e. there exists $a \in \A$ with
\[\phi(a) = ||a||.\]
The set of compactly supported states is denoted $\cStates$.
\end{defn}

By taking self-adjoint part and using functional calculus, we may equivalently require the above norming element $a$ to be such that $0 \leq a \leq I$. One may easily verify that $\phi$ is compactly supported if and only if $\csupp(\phi)$ is compact. Kadison's transitivity theorem implies that any pure state is compactly supported (along with convex combinations of equivalent pure states). We will make use of the following quick lemma for discussion. Recall that an approximate unit $\left(I_n\right)_{n \in N} \subset \A$ is called \emph{almost idempotent} if $I_mI_n = I_n$ for all $m \geq n$. Note that $\sigma$-unital $C^*$-algebras always have almost idempotent - even commuting - sequential approximate units, as follows from a functional calculus argument and the existence of a strictly positive element.

\begin{lemma}\label{lemmacStatesDense}
Given any $\phi \in \cStates$, there exists $a \in \cA$ with $\phi(a) = 1$. If $\A$ contains an almost idempotent approximate unit, then $\cStates \subset \States$ is weak-$*$ dense subset.
\end{lemma}

\begin{proof}
For the first statement, if $\phi(a) = 1$ for $a \in \A$ with $0 \leq a \leq I$, then $\phi(g(a)) = 1$ for any continuous $g:[0,1] \to [0,1]$ which fixes $1$ (which can be seen by passing to the GNS construction of $\phi$, for example). If $g$ is chosen to vanish in some neighbourhood of $0$, then $g(a) \in \saA$ and $0 \leq g(a) \leq I$ (by the non-unital continuous functional calculus), and there exists some continuous $f:[0,1] \to [0,1]$ vanishing at $0$ with $fg=g$, so $f(a) \in \saA$, $0 \leq f(a) \leq I$ and $f(a)g(a) = g(a)$, hence $g(a) \in \cA$.

Suppose $\A$ has an almost idempotent approximate unit $\left(I_n\right)_{n \in N}$. Given any $\phi \in \States$, one must have $\phi(I_n^2) \to 1$. After passing to a subnet is necessary, we can construct a net of states
\[\phi_n(\cdot) = (\phi(I_n^2))^{-1}\phi(I_n \cdot I_n)\]
It is easy to verify that $\phi_n \wsto \phi$, and that $\phi_n(I_{m}) = 1$ for any $m \geq n$, so $\cStates \subset \States$ is weak-$*$ dense
\end{proof}

One is naturally lead to wonder if $\phi \in \cStates$ or $\csupp(\phi)$ is compact for some/all self-similar states associated with a strictly contractive dual IFS.

In the strictly contractive setting, we have already observed that the set $K_\mathbb{F} = \{\phi_\omega; ~\omega \in \infCode{k}\}$ is ``the" $d_L$-compact attractor of the classical strictly contractive IFS $\mathbb{F}^*$ on the classical (extended) complete metric space $(\States, d_L)$ (up to the ever-present issues of having to work with an extended distance function), and that $\phi_\pi$ is contained in the weak-$*$ closed convex hull of $K$ for any $\pi \in \osimp{k}$ (Corollary \ref{corPushforward}).

However, $d_L$-compactness of $K_\mathbb{F}$ does not (necessarily) imply that $\phi_\pi$ is compactly supported for any $\pi \in \osimp{k}$. Indeed, this would be equivalent to the statement that $K_\mathbb{F}$ be contained in a weak-$*$ compact \emph{face} of $\States$ (see \cite{Akemann1971}). Alternatively, that there exists some $a \in \A$ with $0 \leq a \leq I$ with the property that $\phi_\omega(a) = 1$ for all $\omega \in \infCode{k}$. There are easy examples for which compactness of a set of states does not imply containment in a weak-$*$ compact face even in the commutative case, such as the singleton set containing a Gaussian probability measure on $\RR$.

The situation is even worse: we don't even know that each $\phi_\omega$ is compactly supported. In the commutative setting each $\phi_\omega$ is in fact pure, which follows from (weak-$*$ or $d_L$)-closedness of the set of pure states and the fact that adjoints of relatively proper $*$-homomorphisms preserve purity. In the noncommutative setting, closedness of the set of pure states is known to fail (spectacularly) for some $C^*$-algebras (see \cite{Glimm1961}, \cite{Archbold1989}, \cite{Dixmier1977}), and purity-preservation of adjoints of $*$-homomorphisms is subtle (see Corollary 5.8 of \cite{Stormer1968} in conjunction with \cite{BunceChu1998}).

Despite these complications, we make the following conjectures for strictly contractive dual IFS and illustrate some ideas towards their verification.

\begin{conj}
Suppose $\phi_0 \in \cStates$ satisfies the requirements of Section \ref{secStrictlyContractive}. Then if $(\A,L)$ is locally consistent (resp. semi-consistent), each $\phi_\omega$ is pure (resp. compactly supported).
\end{conj}

\begin{conj}
Suppose $\phi_0 \in \cStates$ satisfies the requirements of Section \ref{secStrictlyContractive}.  if $(\A,L)$ is locally semi-consistent, $\phi_\pi$ is compactly supported for each $\pi \in \osimp{k}$.
\end{conj}

Both conjectures boil down to examining the dynamics of $\mB_1$ under $\mathbb{F}$, but only in how they relate to $\phi_0$. Given that $\phi_0$ is compactly supported, there exists some $a \in \cA$ with $\phi_0(a) = 1$. Then $\phi_0(\cdot) = \phi_0(a \cdot a)$ (by Cauchy-Schwarz), so we may alternatively study the dynamics of $\mathbb{F}$ on $\mB_1$ up to multiplying on the left and right by $a$. This hence lands in $a \mB_1 a$, a set which has (by local consistency/semi-consistency) a compact convex quotient by some finite dimensional vector subspace. The dynamics of $\mathbb{F}$ on such a set should be understandable, and allow one to come up with necessary fixed points to show each $\phi_\omega$ and $\phi_\pi$ is compactly supported. In particular, Conjecture \ref{conjMeanErgodic} provides exactly the right type of convergence (albeit for the Ces\`{a}ro averages).

In the case that $(\A, L)$ is locally consistent, purity of each $\phi_\omega$ is expected to follow from $1$-dimensionality of $\locflat{a}$ and the fact that the support projection of $\phi_\omega$ should be contained in a set related to $\locflat{a}$.

Such an endeavour must also rely on algebraic properties (otherwise the maps in $\mathbb{F}$ may as well be positive rather than $*$-homomorphisms, and the result definitely fails for general positive maps), so one may need to suppose in addition that $L$ satisfies some form of Leibniz property (see \cite{Latremoliere2013} and references therein).

\subsection{Ergodic theorems}\

Our Corollary \ref{corPushforward} is one step towards a probabilistic study of self-similar states. There are many other useful probabilistic characterisations of self-similar probability measures and sets related to classical IFS, such as the chaos game (see \cite{Barnsley1989}, \cite{BarnsleyVince2010}) and Elton's ergodic theorem (see \cite{Elton1987}). It is natural to ask if similar things might hold in the noncommutative setting.

Of note, Elton's ergodic theorem only applies on proper metric spaces (i.e. those for which all closed bounded sets are compact). So any noncommutative generalisation will require an extra point in Definition \ref{defnBoundedConsistent} which captures properness.

\subsection{Uniqueness and direct approximation of self-similar projections}\

Classically (e.g. as in \cite{Hutchinson}), one may arrive at the attractor of a strictly contractive IFS $\mathbb{G}$ on some complete metric space $(X,d)$ without any mention of self-similar states. This may be achieved through the use of the Pompeiu-Hausdorff distance, which is complete on the set of non-empty compact subsets of $X$, and is such that the Hutchinson operator, defined by
\[\mathbb{G}(K) := \bigcup_{g \in \mathbb{G}}g(K)\]
for compact $K \subset X$, is contractive.

The research presented in this article began by trying to generalise the above method to the noncommutative setting, but it was discovered that such a construction would not work for various reasons.

If one takes the seemingly reasonable generalisation of the Pompeiu-Hausdorff distance between projections $p, q \in \cProj$, to be the classical Pompeiu-Hausdorff distance (induced by $d_L$) between their associated weak-$*$-closed faces in $\States$, then one runs into issues with Glimm's theorem \cite{Glimm1961}. Namely, if $(\phi_n)_{n = 1}^\infty \subset \States$ is a sequence of pure states converging weak-$*$ to a non-pure state $\phi$, and the convergence occurs in $d_L$ also, then the sequence of weak-$*$ closed faces $\left(\{\phi_n\}\right)_{n=1}^\infty$ cannot converge to a face. Indeed, its Pompeiu-Hausdorff limit is $\{\phi\}$, which is not a face since $\phi$ is not pure.

A more careful approach, based on a dual formulation of the Pompeiu-Hausdorff distance in terms of left-ideals intersecting with $\B_1$ was attempted next, which was (probably) complete but was not flexible enough to allow for an analogue of the Hutchinson operator to be contractive (due to order-theoretic issues).

Hence the project turned to look at self-similar states instead, which has clearly worked much better. The question remains as to whether one can directly approximate self-similar projections in terms of some iterative procedure in $\cProj$.

I.e. Given an extended complete spectral metric space $(\A,L)$, can one put a complete (extended) distance function on $\cProj$ such that given any strictly contractive dual IFS $\mathbb{F}$, the map $\cProj \to \cProj$ defined by
\[p \mapsto \overline{\left(\bigvee_{i=1}^k \uppreimage{(f_i)}(p)\right)}\]
is a contraction?

Note also that uniqueness for self-similar projections does not follow from any of the theorems presented (even if $d_L$ is finite-valued), while it should follow from a direct approximation procedure.

It is quite possible that the Kuperberg and Weaver (see \cite{KuperbergWeaver2010}) approach to noncommutative metric spaces via von Neumann algebras may be more suitable for resolving the above, which has better order theoretic properties than are present in the $C^*$-algebraic framework (albeit with respect to the ``spectral" order). The main reason that the $C^*$-algebraic formulation was favoured over the von Neumann algebraic formulation in the present paper was primarily because fractals and their approximation are topological topics (which do not behave well under passing to almost everywhere equivalence), and secondly because the applications in mind (e.g. Littlewood-Paley-like decompositions for non-Abelian groups) are closer to the $C^*$-algebraic formalism.

\subsection{Applications}\

As mentioned in the introduction, our motivating goal in developing the present theory is to applications in self-similar decompositions similar to the Littlewood-Paley decompositions of function space on $\mathbb{R}^n$.

The theory of self-similar tilings related to classical IFS, as outlined in \cite{BarnsleyVince2014}, is well-developed under certain hypotheses such as requiring attractors to have non-empty interior, or the open set condition (for example, Theorem 3.8 of the aforementioned reference). These requirements have clear analogues in the noncommutative setting, so it would be beneficial to have techniques in order to verify them for particular examples.

Regarding non-empty interior, only special cases are understood even for classical IFS (see \cite{FengFeng2022} and references therein). We, therefore, do not expect to be able to make any general statements regarding the noncommutative case.

There is more that we can say regarding overlap. Suppose that we have some tracial semi-finite normal weight $\tau$ on $\evnA$ (not necessarily faithful, for example, arising as a pullback via some non-degenerate $*$-homomorphism from $\A$ into some semi-finite von Neumann algebra) which is scaled by a factor $S_i \in [0, \infty)$ under each map $f_i$ in some dual IFS $\mathbb{F}$ consisting of invertible maps. Suppose also that we have a self-similar projection $p \in \Proj$ with $\lambda(p) <\infty$. Then we find
\begin{align*}
\tau\left(p\right) &= \tau\left(\bigvee_{i=1}^k f_i^{-1}(p)\right) \\
&\leq \sum_{i=1}^k \tau\left( f_i^{-1}(p)\right) \\
&= \sum_{i=1}^k S_i^{-1} \tau\left(p\right). \\
\end{align*}
If the sum $\sum_{i=1}^k S_i^{-1}$ is equal to $1$, then we have equality
\[\tau\left(\bigvee_{i=1}^k f_i^{-1}(p)\right) = \sum_{i=1}^k \tau\left( f_i^{-1}(p)\right),\]
from which, under repeated application of Kaplansky's formula (\cite{KadisonRingroseII}, Theorem 6.1.7) and the fact that $\tau$ is tracial, we can conclude
\[\tau\left(f_i^{-1}(p) \wedge f_j^{-1}(p)\right) = 0\]
for all $i \neq j$.

A similar statement may not hold for closed self-similar projection, i.e. $p \in \cProj$ with
\[p = \overline{\left(\bigvee_{i=1}^k f_i^{-1}(p)\right)},\]
because the above equality strictly requires the closure on the right hand side (without it, the join may not be closed), and the closure of a projection may be much ``larger" than the initial projection. However, as evidenced by Theorem \ref{thmSupport} and Corollary \ref{corDenseSubprojection}, it is the closed support projections which capture intrinsic information about a dual IFS.

The above computation touches on an important fact: in applications, it is unlikely we will actually care about projections in $\evnA$, and will instead be interested in some quotient (i.e. making $\tau$ above faithful). For example, in studying PDE on Lie groups we will only be interested in the regular representations of the (reduced) group $C^*$-algebra $C^*_r(G)$ on $L^2(G)$, and many projections in the enveloping von Neumann algebra of $C^*_r(G)$ will be identified upon mapping to the regular representations. This leads to yet more regularity questions, such as when the ''boundary" $\partial p = \overline{p} - \mathring{p}$ of some self-similar projection $p$ vanishes in a given quotient. Such a fact would likely be necessary if some kind of self-similar decomposition is to be useful for analysis (see, for example, Chapter 3 of \cite{myPhDThesis}).

Alternatively, we may be able to use results in the regularity theory of projections (see \cite{AkemannEilers2002} and references therein) to control $\partial p$ in a quotient, or the following theorem of Akemann (\cite{Akemann1969}, Theorem II.7) to drop the need for taking closure at the level of $\evnA$ entirely.
\begin{thm}\label{thmJoinClosed}
Suppose $p, q \in \cProj$ are such that
\[||p(q-p \wedge q)|| < 1.\]
Then $p \vee q \in \cProj$.
\end{thm}

\bibliographystyle{plain}
\bibliography{mybib.bib}

\end{document}